\documentclass{amsart}
\usepackage{amsmath,amsfonts,amssymb,amsthm, esint}
\usepackage[english]{babel}
\usepackage{enumerate}
\usepackage{dsfont}
\usepackage{comment}

\usepackage[colorlinks,citecolor=blue]{hyperref}
\usepackage{epsfig,graphicx} 

\newtheorem{theorem}{Theorem}[section]
\newtheorem{definition}[theorem]{Definition}

\newtheorem{lemma}[theorem]{Lemma}
\newtheorem{proposition}[theorem]{Proposition}
\newtheorem{corollary}[theorem]{Corollary}

\theoremstyle{remark}

\numberwithin{equation}{section}
\numberwithin{figure}{section}

\newcommand{\R}{\mathbb{R}}
\newcommand{\N}{\mathbb{N}}

\begin{document}
\author{Tianyu Ma}
	\address{
	HSE University, Moscow, Russian Federation		
}
\email{tma@hse.ru}

	\author{Eugene Stepanov}
	\address{St.Petersburg Branch of the Steklov Mathematical Institute of the Russian Academy of Sciences,
		St.Petersburg, 
		Russian Federation
		\and
		Dipartimento di Matematica, Universit\`a di Pisa,
		Largo Bruno Pontecorvo 5, 56127 Pisa, Italy
		\and
		HSE University, Moscow, Russian Federation		
	}
	\email{stepanov.eugene@gmail.com}
	\date{\today}
	
	\title[eigenvalues and eigenfunctions of MDS defining operators on symmetric spaces]{On eigenvalues and eigenfunctions of the operators defining multidimensional scaling on some symmetric spaces}
	
	\begin{abstract}
	We study asymptotics of the eigenvalues and eigenfunctions of the operators used for constructing multidimensional scaling (MDS) on compact connected Riemannian manifolds, in particular on closed connected symmetric spaces. They are the limits of eigenvalues and eigenvectors of squared distance matrices of an increasing sequence of finite subsets covering the space densely in the limit. We show that for products of spheres and real projective spaces, the numbers of positive and negative eigenvalues of these operators are both infinite. We also find a class of spaces (namely $\mathbb{RP}^n$ with odd $n>1$) whose MDS defining operators are not trace class, and original distances cannot be reconstructed from the eigenvalues and eigenfunctions of these operators. 
	\end{abstract}
	
	\maketitle
	
	\section{Introduction}
	A problem frequently encountered in modern data science is that of reconstructing a metric space $(X, d)$ and the Borel measure
    $\mu$ on it just from the information on the distances between points of a sufficiently large finite subset $\Sigma_k:=\{x_1^k, \dots, x_k^k\} \subset X$. Here we require the subset $\Sigma_k$ to cover $X$ almost densely and with a density approximately $\mu$. Of course, unless $X$ is finite itself, no finite set of points will be sufficient to reconstruct the triple $(X,d,\mu)$ and one can only hope to do this in the limit as $k\to\infty$. To be more precise, we suppose to know distances between points of each set $\Sigma_k$ of some chosen sequence of finite subsets of $X$, and would like to recover from it the information on  $(X,d,\mu)$. This is known as the \textit{learning} problem (or \textit{manifold learning}, when $X$ is a priori supposed to be some smooth, say, Riemannian manifold, and $d$ to be its geodesic distance). 
    
    One of the basic algorithms aimed to solve the learning problem and widely used in applications is \textit{multidimensional scaling} (MDS)~\cite{wang2012geometric}. Although the latter has been originally proposed only for intrinsically Euclidean data
    (i.e., when $X$ is a subset of a Euclidean space $\R^n$ and the distance $d$ is Euclidean), it has been extended to generic metric spaces. Moreover, in applicative ``folklore'', it is often used not only when the distance $d$ is non-Euclidean, but also when $d$ is merely some symmetric function not necessarily satisfying the triangle inequality (the so-called \textit{dissimilarity} function). Whether this application of MDS is justified, i.e., what will be reconstructed by MDS when $d$ is a non-Euclidean distance has been recently posed and solved in~\cite{adams-blumstein-kassab2020MDS} with quite an astonishing answer. Namely, take $X := \mathbb{S}^1$ as a unit circle endowed with its geodesic distance, and let the points of $\Sigma_k\subset X$ to be uniformly spaced so that in the limit as $k\to \infty$ they cover $X$ uniformly. Then according to \cite{adams-blumstein-kassab2020MDS}, MDS yields in the limit as $k \to \infty$ a closed curve in an infinite-dimensional space, which is far from being a circle. An easy calculation shows that it is a fractal object, namely, a \emph{snowflake} embedding \cite{tyson2005characterizations} of a circle in an infinite-dimensional Hilbert space~\cite{PuchSpokSteTrev-manif20}. It becomes an isometric embedding if $\mathbb{S}^1$ is endowed with the geodesic distance 
	raised to some power $\alpha =1/2$. Although this may be unexpected in view of various commonly used applications of MDS, an explanation of this fact may be also traced back to the classical work \cite{von1941fourier} by Neumann and Schoenberg. In their paper, all the invariant metrics on the circle that embed isometrically into a Hilbert space are classified, including of course the $1/2$-snowflake re-obtained via MDS, which is actually credited to the earlier work~\cite{wilson1935certain}, see also the discussion in~\cite[section~7.3]{kassab2019MDS}.
	
	\subsection{Asymptotics of MDS embeddings}\label{sec:MDS_finite1}
	In the study of asymptotical behaviour of the spectra of matrices of squared distances between points of finite samples $\Sigma_k\subset X$, as well
	as of the embedding maps $\mathcal{M}_k\colon \Sigma_k\to\R^k$ produced by MDS, 	the linear operator $T$ over the space $L^2(X,\mu)$ defined by the following formulae plays an important role.
	\begin{equation}\label{eq_KTdef1}
		\begin{aligned}
			K(x,y)&:=- \frac{1}{2} d^2(x, y),\\
		(\mathcal{K} u) (x) &:= \int_X K(x, y) u(y) \,d \mu(y),\\ 
		T & := P \mathcal{K}  P,
	\end{aligned}
	\end{equation}
	where 
	$P$ is the 
	projector operator 
	to the orthogonal complement of constant functions in $L^2(X,\mu)$.  
	Both $\mathcal{K}$ and $T$ are well-defined  under just a mild assumption that $\mu$ have a finite $4$-th order moment, i.e.\
\begin{equation}\label{eq_d4int1}
		\int_{X} d^4(x_0, y) \,d \mu(y) < \infty
\end{equation}
	for some $x_0\in X$ (which holds for instance when $\mu$ is finite and $X$ is bounded). Moreover, in this case, they are self-adjoint Hilbert--Schmidt (and hence compact)  operators.  What is more important is the following: Under the same assumption~\eqref{eq_d4int1},  
suppose the empirical measures $\mu_k$ of finite samples
	$\Sigma_k:=\{x_1^k, \dots, x_k^k\}$ defined by
	\[
	\mu_k:= \frac 1 k \sum_{i=1}^k \delta_{x_i^k},
	\]	
	$\delta_y$ standing for the Dirac mass concentrated in $y\in X$,
	converge to $\mu$ as $k\to\infty$ in the Kantorovich\footnote{Usually, though historically incorrectly, the distances $W_p$ among probability measures are called Wassertsein distances.} $4$-distance $W_4$, i.e.
	$\lim_k W_4(\mu_k, \mu) =0$.  Then according to ~\cite[theorem~5.8]{Stepanov2022},  
	the maps $\mathcal{M}_k$ viewed as functions from $X$ to $\R^\infty$ (with $\R^k$ canonically identified with the subspace of $\R^\infty$ having all zero coordinates except the first $k$ coordinates) converge to some map $\mathcal{M}\colon X\to \R^\infty$, called further \textit{infinite MDS map}, in measure $\mu$ with respect to the product topology on $\R^\infty$ (see also \cite{lim2022classical}). 
The latter is given by the formula
	\begin{equation}\label{eq:MDS1}
		\mathcal{M}(x) := \left(\sqrt{\lambda_1^+} \phi_1^+(x), \sqrt{\lambda_2^+} \phi_2^+(x). \ldots, \sqrt{\lambda_j^+} \phi_j^+(x),\ldots \right), 
	\end{equation}
	where $\lambda_1^+ \ge \lambda_2^+ \ge \dots > 0$ are positive eigenvalues of $T$ (counting multiplicity), and $\{\phi_j^+\}_j \in L^2(X,\mu)$ is an orthonormal system in $L^2(X,\mu)$ made of the respective eigenfunctions, i.e.\ $T\phi_j^+=\lambda_j^+\phi_j^+$. Note that the definition of $M$ depends on the choice of $\phi_j^+$. Here we silently assume that if the set of positive eigenvalues of $T$ contains $N< \infty$ elements, then $(M(x))_j := 0$ for $j >N$. 
	
	By calculating explicitly the eigenvalues and eigenfunctions of $T$ and using~\eqref{eq:MDS1}, one shows in~\cite{Stepanov2022} that $M$ gives a snowflake (Assouad-type) embedding of any $m$-dimensional sphere $\mathbb{S}^m$ or any   
	 $m$-dimensional flat torus $(\mathbb{S}^1)^m$ into the Hilbert space $\ell^2$ of square summable sequences (in the calculations one assumes
	 $\mu$ to be the respective volume measure in all these cases).
	
	Another important observation is the following: With the stronger assumption that $T$ is a trace-class operator, i.e.,
	\begin{equation*}\label{eq_Tnucl11}
		\sum_i |\lambda_i| < +\infty,
	\end{equation*} 
	where $\{\lambda_i\}$ stands for the sequence of \textit{all} eigenvalues of $T$, let the metric measure space $(X,d,\mu)$ be, say, infinitesimally doubling~see~\cite[theorem~3.4.3]{heinonen2015sobolev} (which includes any smooth Riemannian manifold equipped with geodesic distance and volume measure). Then by~\cite[theorem~5.8]{Stepanov2022} (see also \cite{lim2022classical}), in a sense,
	distances between almost every pair of points can be recovered from the spectrum of $T$ and the set of the respective eigenfunctions. Namely, in this case, we have
	\begin{equation}\label{eq_aeinj11}
		\sum_{i = 1}^\infty \lambda_i \left(\phi_i(x) - \phi_i(y)\right)^2 = 
		d^2(x, y) \quad 	\text{for $\mu \otimes \mu$-a.e.\ $(x,y)$},
\end{equation}
where $\{\phi_i\}$ stands for an orthonormal basis in $L^2(X,\mu)$ made of eigenfunctions of $T$ with $T\phi_i=\lambda_i \phi_i$. The importance of the trace class condition on the operators $T$ and $\mathcal{K}$ for the asymptotics of the spectra of distance matrices has been also studied recently in~\cite{vershik2023limit}.
	
	\subsection{Questions and results}\label{sec:MDS_questres1}
	The above-cited results raise a series of curious questions. Namely, one asks whether there are natural examples of spaces $(X,d,\mu)$ 
	such that
	\begin{itemize}
		\item[(Q1)] 
		no infinite MDS map  (i.e., independently on the choice of eigenfunctions of $T$) gives a topological embedding of $X$ into $\ell^2$,
		\item[(Q2)] the operator $T$ is not trace-class and/or the distance reconstruction formula~\eqref{eq_aeinj11} is not valid.
	\end{itemize}
	We find both examples among just compact Riemannian manifolds with volume measure, namely, both $(Q1)$ and $(Q2)$ are satisfied by projective spaces of sufficiently high dimension.
	We do so by studying the operator $T$, its eigenvalues and eigenfunctions for symmetric compact Riemannian manifolds with volume measure. Note that in general, there seems to be no easy way to find either the spectrum or eigenfunctions of $T$. However, in this particular case
	the situation greatly simplifies since we are able to show that $T$ commutes with the Laplace-Beltrami operator, which allows us to search for its eigenfunctions among the eigenfunctions of the latter. 
	We are able then to show that if $X$ is a finite product of spheres of any dimensions, the infinite MDS map gives a snowflake embedding of
	$X$ into $\ell^2$ thus generalizing the results of~\cite{Stepanov2022}, while if $X$ is a projective space with sufficiently high dimension, then $M$ does not send $X$ to $\ell^2$ at all, and in particular the distance reconstruction formula~\eqref{eq_aeinj11}
	is not valid. Curiously however, it happens that in all these cases. Moreover, if $X$ is a finite product of spheres and projective spaces, then the spectrum of $T$ contains infinitely many positive and negative eigenvalues. This contrasts with the case that $(X,d)$ is isometrically embeddable in a Hilbert space, in which all the eigenvalues of $T$ are positive.

\section{Notation and preliminaries}

For vectors $x$ and $y$ in the Euclidean space $\R^n$, we denote by $x\cdot y$ their Euclidean scalar product. The Euclidean norm is denoted by  $|\cdot|$
Let $\ell^2$ be the usual Banach space of square summable sequences equipped with its usual norm  $\|\cdot\|_2$. The space $\R^\infty$ stands for the linear
space of all real-valued sequences (sometimes denoted by $\R^\N$ in the literature), equipped with its product topology. The norm $\Vert \cdot\Vert_2$ can be extended to a ``metric" on $\mathbb{R}^{\infty}$ taking values in $[0,+\infty]$.

If $X$ is a smooth Riemannian manifold, we denote by $C^\infty(X)$ the set of infinitely smooth functions over $X$. 

Throughout the paper, we sometimes use the big Theta notation by D.~Knuth.

For a metric measure space $(X,d,\mu)$, we will assume $\mu$ to be a Borel probability measure. By $\langle x.y\rangle$ we denote then the standard scalar product in the Hilbert space $L^2(X,\mu)$. For a $u\in L^2(X,\mu)$ we let $u^\perp$ stand for its orthogonal complement in $L^2(X,\mu)$. The spectrum of a linear operator $T$ counting multiplicity is denoted as $Spec(T)$, and its signature $sgn(T)$ consists of three values ordered in the numbers of zero, positive and negative eigenvalues. 

Recall for a metric measure space $(X,d,\mu)$, the MDS map $\mathcal{M}: X\rightarrow \mathbb{R}^{\infty}$ defined as in \eqref{eq:MDS1} are obtained from the positive eigenvalues and their corresponding eigenfunction the operator $T$. Such maps are not unique, a different choice of the orthonormal set $\{ \phi_j^+\}$ yields a different map. However, they  all have the common property given by the following 
statement.

\begin{lemma}
\label{lemma:mds dist inv}
If the operator $T$ is Hilbert-Schmidt, then for any MDS maps $\mathcal{M}^1,\mathcal{M}^2$, we have
\begin{align*}
	\Vert \mathcal{M}^1(x)-\mathcal{M}^1(y)\Vert_2=\Vert \mathcal{M}^2(x)-\mathcal{M}^2(y)\Vert_2,\quad \mbox{for all } x,y\in X.
\end{align*}
Moreover, in this case the right-hand side of~\eqref{eq_aeinj11} is independent of the choice of eigenfunctions $\phi_j$ of $\mathcal{K}$. 
\end{lemma}

\begin{proof}
Since $T$ is compact, every non-zero eigenspace of $T$ is finite-dimensional. A new choice of the orthonormal set of $\{\phi_j^+\}$ is obtained from $L^2(X,\mu)$-orthogonal transformation of each eigenspace $E_{\lambda_{\alpha}^+}$, with the eigenvalue $\lambda_{\alpha}^+>0$. Let $\{ \phi_j^{\alpha}\}$ and $\{ \tilde{\phi}_j^{\alpha} \}$ be two orthonormal bases of $E_{\lambda_{\alpha}^+}$. The space $E_{\lambda_{\alpha}^+}$ with $L^2$-norm can be identified as a Euclidean space with the standard norm, and orthogonal transformation on Euclidean spaces preserves Euclidean distances. Thus, we have
\begin{align*}
\sum_j \left(  \phi_j^{\alpha}(x)-\phi_j^{\alpha}(y)    \right)^2= \sum_j \left(  \tilde{\phi}_j^{\alpha}(x)-\tilde{\phi}_j^{\alpha}(y)    \right)^2
\end{align*}
Combining the equation above with the defining formula of MDS maps in \eqref{eq:MDS1}, we obtain the desired equality.
The independence of the right-hand side of~\eqref{eq_aeinj11} on the choice of eigenfunctions $\phi_j$ of $\mathcal{K}$ is shown in the same way. 
\end{proof}

\subsection{Review on Riemannian symmetric spaces}
\label{sec:Riem sym prelim}
We give a brief review of the basic properties of Riemannian symmetric spaces in this section. We begin by recollecting the basic facts on Riemannian manifolds which will be used in this paper. The metric $g$ of a connected Riemannian manifold $(M^n,g)$ defines a volume measure $\mu$ and a distance function $d$ on $M$, where
\begin{align*}
	d(x,y)=\inf\lbrace \mathrm{length}(\gamma): \gamma\ \mathrm{is\ a\ curve\ from}\ x\ to\ y \rbrace.
\end{align*}
This yields a metric measure space $(M,d,\mu)$. By Hopf-Rinow theorem, $(M,d)$ is a complete metric space if $(M,g)$ is geodesically complete. If $(M,g)$ is complete and connected, then for any $x,y\in M$, there exists a distance minimizing geodesic. For any $x\in M$, the tangential cut locus at $x$ is the set of $v\in T_xM$ such that $\exp(tv)$ is a minimizing geodesic for $t\in [0,1]$, but fails to be a minimizing geodesic for any $t>1$. The cut locus $\mathcal{C}_x$ at $x$ is the image of this set under the exponential map at $x$. If $y\in\mathcal{C}_x$, we have either $y$ is conjugate to $x$ or there is more than $1$ distance minimizing geodesic from $x$ to $y$ \cite[lemma~ 8.2]{Petersen1998}.  Therefore, we have $y\in\mathcal{C}_x$ if and only if $x\in\mathcal{C}_y$. Define the symmetric subset of $M\times M$ by
$$\mathcal{C}=\lbrace (x,y)\in M\times M: x\in\mathcal{C}_y\rbrace$$
It is a well-known fact that the function $d^2(x,\cdot)$ is smooth outside $\mathcal{C_x}$. 

By a Laplacian on the Riemannian manifold, we always mean the Laplace-Beltrami operator.

Let $(M,g)$ be a Riemannian manifold. Recall that a local geodesic symmetry at $p\in M$ is a local diffeomorphism $r_p$ on a neighbourhood of $p$ such that for all geodesic $\gamma(t)$ with $\gamma(0)=p$, we have $r_p(\gamma(t))=\gamma(-t)$. 
\begin{definition}
	A Riemannian manifold $(M,g)$ is called a symmetric space if for every $p\in M$, the local geodesic symmetry $r_p$ can be extended to a global isometry on $M$ fixing $p$.
\end{definition}
It is obvious from the definition that all symmetric spaces are complete, and all connected symmetric spaces are homogeneous, i.e., the isometry group $G$ acts transitively on $(M,g)$. In fact, every connected symmetric space is a reductive homogeneous space as follows. Denote $G$ and $K$ the isometry group and the isotropy of some $p\in M$, respectively. Let $\mathfrak{g}$ and $\mathfrak{k}$ be the corresponding Lie algebras of $G$ and $K$. The geodesic symmetry $r_p\in G$ at $p$ satisfies $r_p^2=Id$. Hence $Ad(r_p)$ is an involutive Lie algebra automorphism of $\mathfrak{g}$. The Lie algebra $\mathfrak{g}$ admits a decomposition $\mathfrak{g}=\mathfrak{k}+\mathfrak{m}$ such that $\mathfrak{k}$ and $\mathfrak{m}$ are the eigenspaces of $Ad(r_p)$ corresponding to the eigenvalues $1$ and $-1$, respectively.

We close this section with some geodesic properties of symmetric space which will be used later in the proof. Let $\gamma: \mathbb{R}\rightarrow M$ be a complete geodesic for the symmetric space $(M,g)$. Then the following family of composed map $\tau_s=r_{\gamma(s)}\circ r_{\gamma(0)}$ is a 1-parameter subgroup of the isometry group $G$. 
Easily we have 
\[\tau_s(\gamma(t))=\gamma(t+s).\] 
Such maps are called \textit{geodesic transvections} along $\gamma$. In terms of the cut locus, although determining the geometry of cut loci of a general Riemannian manifold can be difficult, the cut loci of symmetric spaces have been well studied, see~\cite{Helgason2001,Sakai1978,Takeuchi1978}. In particular, for a compact symmetric space $(M^n,g)$, the cut locus at any $x\in M$ is a finite disjoint union of regular submanifolds with possible different dimensions \cite[theorem~3.3]{Takeuchi1978}. Taking the union of the submanifolds of dimension $n-1$ in this decomposition of $\mathcal{C}_x$, the Riemannian volume density $\mu$ together with the perpendicular unit vector fields define a measure $\underline{\mu}_x$ on $\mathcal{C}_x$. Hence, we may view $\mathcal{C}_x$ as a ``piece-wise smooth" manifold of dimension $n-1$. Note that $\underline{\mu}_x$ can be zero if the decomposition of $\mathcal{C}_x$ has no components of dimension $n-1$, for example when $M$ is the standard sphere.

\section{The MDS map for closed connected symmetric spaces}
\label{sec: sym space commute operator}
From now on, we focus on the MDS maps of closed connected symmetric spaces. Let $(S,g)$ be a closed connected symmetric space. Denote $d$ and $\mu$ the distance function and Borel measure induced by $g$ as before (i.e., the geodesic distance and the Riemannian volume measure). Since $S$ is compact, it is well known that there is an orthonormal basis of $L^2(S)$ contained in $C^{\infty}(S)$ consisting of Laplacian eigenfunctions, see e.g., \cite[p.~2]{zelditch2017eigenfunctions} and \cite[theorem~2.2.17]{levitin2023topics}. In addition, each eigenvalue of $\Delta$ has finite multiplicity, and eigenspaces of distinct eigenvalues are mutually orthogonal.

Several easy consequences follow from our assumption. Since $S$ is compact, the integral kernel of $\mathcal{K}$ is bounded and uniformly continuous on $S\times S$. Thus by Cauchy–Schwarz inequality, each eigenfunction of $\mathcal{K}$ (also for $T$) is continuous. In addition, for each non-zero eigenvalue of $\mathcal{K}$, the corresponding eigenspace is finite-dimensional.

As both the Laplacian operator and the integral kernel $K(x,y)$ are closely related to the distance function, we would like to establish a relation between them. We begin with the following lemma leading to the symmetric property of the integral kernel of $\mathcal{K}$ for symmetric spaces.

\begin{lemma}
	\label{lemma:K-D sym 1}
	Let $\mu\otimes\mu$ be the product measure on $S\times S$. For a compact symmetric space $S$, the integral kernel of $\mathcal{K}$ satisfies
	\begin{align}
		\label{eqn: pt laplace sym}
		\Delta_x K(x,y)=\Delta_y K(x,y), \mathrm {for\ a.e.}\ (x,y)\in S\times S
	\end{align}
\end{lemma}
\begin{proof}
	First note that the integral kernel $K(x,y)=-\dfrac{1}{2} d^2(x,y)$ is always symmetric, i.e. $K(x,y)=K(y,x)$. For a fixed $x\in S$, the function $K(x,\cdot)$ is smooth on $S\setminus \mathcal{C}_x$, which is an open set of full measure in $S$. Since $x\in C_y$ if and only if $y\in \mathcal{C}_x$, the functions $\Delta_x K(x,y)$ and $ \Delta_y K(x,y)$ are well defined a.e. on $S\times S$ with respect to $\mu\otimes \mu$.
	
	Suppose that $x\notin C_{y}$, and let $\gamma:[-1,1]\rightarrow S$ be the unique minimizing geodesic from $x$ to $y$. The geodesic symmetry $r_0$ at the median $\gamma(0)$ is an isometry interchanging $x$ and $y$. Denote $K_y$ the function $K(\cdot,y)$. Since $r_0$ is an isometry, in a neighbourhood $U$ of $y$ we have
	$$(\Delta K_y)(r_0(z))=\Delta (K_y\circ r_0)(z),\quad \mbox{for all }z\in U.$$
	Denote $\Delta_1, \Delta_2$ the Laplacian with respect to the first and second coordinates in $S\times S$, respectively. We also have
	\begin{align*}
		\Delta (K_y\circ r_0)(z)=\Delta_1 K(r_0(z),y)=\Delta_1 K(z,r_0(y))=\Delta_1 K(z,x)=\Delta_2 K(x,z).
	\end{align*}
	Taking $z:=y$, we get the desired equality.
\end{proof}
The following proposition shows that the operator $T$ commutes with the extension of the  Laplacian operator $\Delta^D$.

\begin{proposition}
	The operator $T$ commutes with Laplacian $\Delta$ on closed connected symmetric spaces in the following sense,
	\begin{align}
	\label{eqn:Laplace commute1}
	\langle Tf, \Delta h\rangle =\langle T\Delta f,h\rangle,\quad \mbox{for all } f,h\in C^{\infty}(S).
	\end{align}
	Thus $T$ preserves each Laplacian eigenspaces.
\end{proposition}
\begin{proof}
	Since the only harmonic functions on closed manifolds are constants, the Laplacian operator $\Delta$ commutes with the projection $P$ in \eqref{eq_KTdef1} when acting on smooth functions. As $T=P\mathcal{K}P$, it suffices to show $\mathcal{K}$ commutes $\Delta$ by the following equality.
	\begin{align}
	\langle \mathcal{K} f, \Delta h\rangle =\langle \mathcal{K}\Delta f,h\rangle,\quad \mbox{for all } f,h\in C^{\infty}(S).
	\end{align}
	\\
	Firstly, from Lemma \ref{lemma:K-D sym 1}, we have the following observation on symmetric spaces.
	\begin{align*}
		\Delta_x K(x,y)=\Delta_y K(x,y), \mathrm {for\ a.e.}\ (x,y)\in S\times S
	\end{align*}
	For any $x\in S$, denote $\mathcal{C}_x$ the cut locus of the function $K(x,\cdot)$ as before. We cut the manifold $S$ from $\mathcal{C}_x$, and obtain a manifold $M_x$ with boundary $\partial M_x$. Let $i_x: \partial M_x\rightarrow \mathcal{C}_x$ be the canonical projection on the boundary. Using the decomposition theorem of $\mathcal{C}_x$ mentioned in Section \ref{sec:Riem sym prelim} (see \cite{Takeuchi1978} for details), we obtain a decomposition of $\partial M_x$. Each component $\mathcal{C}_x(i)$ of dimension $n-1$ in $\mathcal{C}_x$ is a regular submanifold, and its preimage under $(i_x)^{-1}$ is a double cover $\partial M_x^0(i)\sqcup \partial M_x^1(i)$. Thus, by restricting only to components of $\partial M_x$ of dimension $n-1$, the unit outer normal $N_x$ defines a measure $\underline{\mu}^x$ on $\partial M_x$. The space $(\partial M_x,\underline{\mu}^x)$ is isomorphic to $(\mathcal{C}_x,\underline{\mu}_x))\oplus (\mathcal{C}_x,\underline{\mu}_x)$ as measure spaces. For a compact symmetric space $(S^n,g)$, the conjugate locus at any $x\in S$ has dimension strictly less than $n-1$ \cite[theorem~7.3.3]{Helgason2001}. Since $\exp_x$ is non-singular outside the tangential conjugate locus, $\nabla_y K(x,y)$ is well-defined a.e. on $\partial M_x$.

	For all $f,h\in C^{\infty}(S)$, we can apply the divergence theorem to obtain
	\begin{align*}
		\langle \Delta^D \mathcal{K} f,h\rangle
		=&\int_S f(y)\int_S \Delta_x h(x) K(x,y) d\mu(x) d\mu(y)\\
		=& \int_S f(y)\int_{\partial M_y} K(x,y)\langle \nabla_x h,N^y(x) \rangle d\underline{\mu}^y(x) d\mu(y)\\
		& -\int_S f(y) \int_S \langle \nabla_x h, \nabla _x K(x,y)\rangle d\mu(x) d\mu(y)
	\end{align*}
	Note that for each $y\in S$, the integral
	$$\int_{\partial M_y} K(x,y)\langle \nabla_x h,N^y(x) \rangle d\underline{\mu}^y (x) $$
	vanishes because $h$ is smooth on $S$. Therefore, we get
	\begin{align}
		\label{eqn:comm lhs}
		\langle \Delta^D \mathcal{K} f,h\rangle =  -\int_S f(y) \int_S \langle \nabla_x h, \nabla _x K(x,y)\rangle d\mu(x) d\mu(y).
	\end{align}
	Analogously, we obtain
	\begin{align}
		\label{eqn:comm rhs}
		\langle \mathcal{K} \Delta f,h\rangle  = -\int_S h(x)\int_S \langle \nabla_y f,\nabla_y K(x,y)\rangle d\mu(y)d\mu(x)
	\end{align}
	On the other hand, from \eqref{eqn: pt laplace sym} and the divergence theorem, we get
	\begin{align*}
		& \int_S h(x)\left( \int_{\partial M_x}f(y)\langle \nabla_y K,N^x_y\rangle d\underline{\mu}^x(y)-\int_S \langle \nabla_y f, \nabla_y K\rangle d \mu(y) \right) d\mu(x)\\
		= & \int_S f(y) \left( \int_{\partial M_y}h(x)\langle \nabla_x K,N^y_x\rangle d\underline{\mu}^y(x)-\int_S \langle \nabla_x h, \nabla_x K\rangle d \mu(x) \right) d\mu(y)
	\end{align*}
	Comparing with the $(\ref{eqn:comm lhs},\ref{eqn:comm rhs})$ obtained earlier, we need to prove
	\begin{equation}
		\begin{aligned}
			\label{eqn:divergence sym 1}
			& \int_S h(x) \int_{\partial M_x}f(y)\langle \nabla_y K,N^x_y\rangle d\underline{\mu}^x(y) d\mu(x)\\
			=&   \int_S f(y)  \int_{\partial M_y}h(x)\langle \nabla_x K,N^y_x\rangle d\underline{\mu}^y(x) d\mu(y)
		\end{aligned}
	\end{equation}
	For any $y\in \mathcal{C}_x(i)$ not conjugate to $x$, let $y_j\in M^j_x(i)$ be the preimage of $y$ for $j=0,1$.  Since distance minimizing geodesics segments starting from $x$ cannot intersect $\mathcal{C}_x$ except for the end-points, there exists distance minimizing unit speed geodesics $\gamma_j$ in $M_x$ from $x$ to $y_j$ for $j=0,1$. For simplicity, denote the geodesic in $S$ corresponding to $\gamma_j$ by the same notation.
	The geodesic transvection $\tau_j$ along $\gamma_j$ sending $x$ to $y$ is an isometry. It maps $\mathcal{C}_x$ to $\mathcal{C}_y$, and the curve $\gamma_j$ into itself with $\tau_j\circ r_x(y)=x$. Thus, the component of $\mathcal{C}_y$ containing $x$ is a regular submanifold of dimension $n-1$, and the preimage $i^{-1}_y(x)$ also contains exactly two points in $\partial M_y$. Let $x_j$ be induced by the reversed curve of $\gamma_j$. Because $y_j$ is not conjugate to $x$ and $x_j$ is not conjugate to $y$, we have
	$\nabla_y K(x,y_j)$ and $\nabla_x K(y,x_j)$ are well-defined.
	If we can show
	\begin{align}
		\label{eqn: angle symmetry}
		\sum_{j=0}^1 \langle\nabla_y K(x,y_j),N^x_{y_j} \rangle=
		\sum_{j=0}^1 \langle\nabla_x K(y,x_j),N^y_{x_j} \rangle,
	\end{align}
	then Equality (\ref{eqn:divergence sym 1}) is just an application of the Fubini theorem. Since $\gamma_j$ is a distance minimizing geodesic segment  from $x$ to $y_j$, we have  \begin{align*}
		\nabla_y K(x,y_j)=- d(x,y)\dot{\gamma_j}(d(x,y))
	\end{align*}
	Combining the equality above with the fact $\tau_j\circ r_x$ is an isometry mapping $y$ to $x$, we obtain
	\begin{align*}
		(\tau_j\circ r_x)_*(\nabla_y K(x,y_j))=\nabla_x K(x_j,y),\ (\tau_j\circ r_x)_*(N^x_{y_j})=N^y_{x_j},\quad \mbox{for all } j=0,1
	\end{align*}
	Then the equation \eqref{eqn: angle symmetry} holds. Hence, the Laplacian operator commutes with the integral operator $T$ for closed symmetric spaces in the sense of \eqref{eqn:Laplace commute1}.\\
	\\
	Finally, note that for all $f_1,f_2\in C^{\infty}(S)$ such that $\Delta f_i=\lambda_i f$ with $i=1\ \mathrm{and}\ 2$, we have
\begin{align*}
\langle Tf_1,\lambda_2 f_2\rangle=\langle Tf, \Delta f_2\rangle=\langle T\Delta f_1,f_2\rangle =\langle \lambda_1 Tf_1,f_2\rangle.
\end{align*}
Therefore, $Tf_1$ is perpendicular to all eigenspaces of $\Delta$ corresponding to eigenvalues distinct from $\lambda_1$. Since there is an orthonormal basis of $L^2(S)$ consisting of $\Delta$-eigenfunctions, we see
$Tf_1$ is contained in the eigenspace of $\lambda_1$ for $\Delta$. This completes the proof.
\end{proof}

\section{MDS for elementary symmetric spaces}
\subsection{Products of spheres}
We can derive the spectrum of $T$ for product spaces based on the components in the decomposition. This generalizes the discussion of the product of two spaces given in Section 6.2 of \cite{Stepanov2022}. Let $(X,d,\mu)$ be the metric measure space induced by the product manifold 
\begin{align*}
	X=\prod_{i=1}^N X_i,\ g=\sum_{i=1}^N g_i;
\end{align*}
where each $g_i$ is the Riemannian metric on a compact connected manifold $X_i$. Denote $d_i$ and $\mu_i$ the Riemannian distance and the normalized Riemannian volume on $(X_i,g_i)$. It follows that for $x=(x_i)\in X$ and $y=(y_i)\in X$, we have
\begin{align*}
	\mu=\prod_{i=1}^N \mu_i,\  d^2(x,y)=\sum_{i=1}^N d_i^2(x_i,y_i).
\end{align*}
Since the space $(X_i,g_i)$ are all connected, the Hilbert spaces $L^2(X,\mu)$ and $L^2(X_i,\mu_i)$ are all separable. One would naturally expect the MDS maps and the associated operator to be in the form of Cartesian products of those for each component. This is exactly the case as in \cite[section~6.2]{Stepanov2022}. Suppose $\{\phi_i^j\}$ is an orthonormal basis diagonalizing the associated operator $T_i$ for $(X_i,d_i,\mu_i)$. Then $\{ \phi_{i_1}^{j_1}\cdots\phi_{i_N}^{j_N}\}$ is an orthonormal basis for $L^2(X,\mu)$ by Fubini theorem.  According to \cite[proposition~6.2]{Stepanov2022}, the associated operator $T$ is diagonalizable with respect to this basis. Since the constant functions are in the kernel of $T_i$ for all $1\leq i\leq N$, we can choose $\{ \phi_i^j\}$ such that the non-constant elements are in $1^{\perp}$. Then for $\Phi= \phi_{i_1}^{j_1}\cdots\phi_{i_N}^{j_N}$, we have $T(\Phi)=0$ unless the components of $\Phi$ contain exactly 1 non-constant function. To see this, we first compute
\begin{align*}
	-2K(\Phi)(x)=\sum_{i=1}^N \int_X d_i^2(x_i,y_i) \phi_{i_1}^{j_1}(y_1)\cdots \phi_{i_n}^{j_N}(y_N) d\mu(y)
\end{align*}
If there is more than one non-constant component in $\Phi$, apply Fubini theorem to the integrals above. We can re-order the integration such that inside each integral, the inner most term satisfies
$$\int_{X_l} \phi_l^{j_l} du_l=0$$
Therefore, if $N>1$, the kernel of $T$ is infinite dimensional.

If all components of $\Phi$ are constant, clearly $T(\Phi)=P\mathcal{K}P(\Phi)=0$.
Assume $\Phi=\phi_1^{j_1}$ is non-constant and $T_1(\phi_1^{j_1})=\lambda_1^{j_1} \phi_1^{j_1}$. We obtain $T(\Phi)=\lambda_1^{j_1} \Phi$. As a result, for the non-zero spectra of $T$ and $T_i$, we have
\begin{align}
\label{eqn:prod spec1}
	Spec(T)\setminus \{ 0\} = \sqcup_{i=1}^N \left( Spec(T_i) \setminus \{ 0\}\right),
\end{align}
where both sides of the relation above are eigenvalues counted with multiplicity. This implies $T$ is a trace-class operator if and only if  each $T_i$ is trace-class.

Since the map $T$ for any compact space has at least 1 positive eigenvalue, the product of $N$ Riemannian manifolds shall have at least $N$ positive eigenvalues counting multiplicity. Summing up, we have the following assertion.
\begin{proposition}
	Let $X=\prod_{i=1}^N X_i$ and $g=\sum_{i=1}^N g_i$ be the product space of compact connected Riemannian manifolds. Denote $d$ and $\mu$ the distance and measure induced by $g$. Then the map $T$ for $(X,d,\mu)$ has at least $N$ positive eigenvalues counting multiplicity, and $T$ is a trace-class operator if and only if each $T_i$ is trace-class.
\end{proposition}
From now on, let $X$ be a product of spheres. We know from \cite[section~6.1]{Stepanov2022} the operator $T_i$ for each component sphere is trace-class and has signature $(1,\infty,\infty)$. Combining 
  \eqref{eqn:prod spec1} and the proposition above, we obtain the following statement.
 \begin{corollary}
 Let $X$ be a finite product of $N$ spheres (including circles), and define the distance $d$ and measure $\mu$ as above. The operator $T$ for $(X,d,\mu)$ is trace-class and has signature $(\infty,\infty,\infty)$ for $N>1$.
 \end{corollary}
As $X$ is closed symmetric, the operator $T$ preserves all eigenspaces of $\Delta$. Let $E_{\lambda_i^{\alpha}}$ be an eigenspace for the Laplace operator of $(X_i,g_i)$ and denote $\pi_i:X\rightarrow X_i$ the canonical projection,
then $T$ preserves the space $\pi_i^*\left(E_{\lambda_i^{\alpha}}\right)$ and is self-adjoint on it with respect to the $L^2(X,\mu)$-norm. In fact, the operator $T$ is simply a constant scaling on $\pi_i^*\left(E_{\lambda_i^{\alpha}}\right)$. To show this, denote $G$ the isometry group of $g$, then for any $s\in G$ we have
\begin{align*}
	\mathcal{K}(s\cdot f)(x)& =\int_X K(x,y) f(s\cdot y) d\mu(y)\\
	& = \int_X K(s\cdot x,s\cdot y) f(s\cdot y) d\mu(y)\\
	& = \int_X K(s\cdot x,z) f(z) d\mu(z)= \mathcal{K}(f)(s\cdot x).
\end{align*}
Therefore, the natural action of $G$ commutes with the operator $\mathcal{K}$ (hence $T$). Let $G_i$ be the isometry group of $(X_i,g_i)$. Since each component $(X_i,g_i)$ is a sphere, the isometry group $G_i$ acts irreducibly on $E_{\lambda_i^{\alpha}}$  \cite[theorem~3.1]{Helgason1984}, and preserves the eigenspaces of $T$ when viewed as a subgroup of $G$. Thus, $T$ can have only 1-real eigenvalue on $\pi_i^*\left(E_{\lambda_i^{\alpha}}\right)$. In particular, the map $T$ has only 1 real eigenvalue on each Laplacian eigenspace of standard spheres.

\subsection{Projective spaces}
\label{sec: proj spaces}
\subsubsection{Signature of $T$ for projective spaces}
Here we consider the case $X:=\mathbb{RP}^n$, an $n$-dimensional real projective space equipped with its geodesic distance and Riemannian volume measure.  Since $\mathbb{RP}^n=\mathbb{S}^n/\mathbb{Z}_2$, we expect the MDS for projective spaces behaves similarly compared to spheres. On the other hand, we will see in this section how the global topology makes a difference in MDS maps on projective spaces compared to the MDS on spheres. We start by determining the signature of $T$ for projective spaces $\mathbb{RP}^n$.

The Laplacian eigenfunctions on $\mathbb{RP}^n$ are well-defined projections of spherical harmonics. Hence, they are projections of spherical harmonics of even degree. Clearly as for the spheres, the group $SO(n)$ acts irreducibly on each Laplacian eigenspace of $\mathbb{RP}^n$. Thus, each eigenspace of $T$ is a direct sum of eigenspace of $\Delta$ on $\mathbb{RP}^n$.

For $x,y\in\mathbb{S}^n$, the distance between the lines $[x]$ and $[y]$ on $\mathbb{RP}^n$ is $\arccos(\vert\langle x,y\rangle\vert)$. We know that the operators $T$ and $\mathcal{K}$ share all eigenvalues and eigenfunctions except for those corresponding to constant functions. To compute the spectrum of $\mathcal{K}$ for projective spaces, let $2k$ be even. Then by Funke-Hecke theorem \cite[p.~98]{Groemer1996}, the eigenvalue of $\mathcal{K}$ for $\mathbb{RP}^n$ corresponding to spherical harmonics of degree $2k$ is given by
\begin{align}
	\label{eqn:proj spec}
	\lambda_{2k}^n=\sigma_n \int_0^1\ \arccos^2(t) P_{2k}^n(t) (1-t^2)^{(n-2)/2} dt,
\end{align}
where $P_{2k}^n(t)$ is the Legendre polynomials for $\mathbb{S}^n$ of degree $2k$, and the numbers $\sigma_n= -\dfrac{\mathrm{vol}(\mathbb{S}^{n-1})}{\mathrm{vol}(\mathbb{S}^n)}$ are negative constants depending only on $n$. 
Inside the integral in $\eqref{eqn:proj spec}$, only the term
\begin{align}
	\label{eqn:Reccurence explicit}
	F_{2k}^n(t)=P_{2k}^n(t)(1-t^2)^{(n-2)/2}= R_{2k}^n \left( \dfrac{d}{dt} \right)^{2k} (1-t^2)^{2k+(n-2)/2}
\end{align}
is affected by the dimension $n$ and even degree $2k$. For $2k$ even, the Rodrigues constants $R_{2k}^n$ are given by~\cite[p.~22]{Muller}
\begin{align}
	R^n_{2k}=\dfrac{1}{4^k}\dfrac{\Gamma(n/2)}{\Gamma(n/2+2k)}.
\end{align}
These constants are all positive. Note the equations (\ref{eqn:proj spec},\ref{eqn:Reccurence explicit}) are still valid for the special case $n=1$, where $P_{2k}^1$ are given by Chebyshev polynomials of the first kind. 
For $n\geq 1 $ and $2k\geq 2$, a substitution $t=\cos\theta$ using $\dfrac{d}{d\theta}=-\sin\theta\cdot\dfrac{d}{dt}$ yields
\begin{align*}
	& F_{2k}^n(\cos\theta)\\
	=& R_{2k}^n \left(\dfrac{-1}{\sin\theta}\dfrac{d}{d\theta}\right)^{2k}  \left[(\sin\theta)^{4k+n-2}\right] \\
	=& (2k+n-2)R_{2k}^n \left(\dfrac{-1}{\sin\theta}\dfrac{d}{d\theta}\right)^{2k-2} \biggr[(2(2k-1)+n-2)(\sin\theta)^{2(2k-2)+n-2}\\
	& -(2(2k-1)+s+1) (\sin\theta)^{2(2k-2)+n}  \biggr]\\
	=& (4k+n-2)\left[(2(2k-1)+n-2)\dfrac{R^n_{2k} }{R^n_{2k-2}}F_{2k-2}^n(\cos\theta)\right.\\
	& \left. -(2(2k-1)+n-1)\dfrac{R_{2k}^n }{R^{n+2}_{2k-2}}F_{2k-2}^{n+2}(\cos\theta) \right]
\end{align*}
By linearity of the integration, we combine the equation above and \eqref{eqn:proj spec} to obtain
\begin{equation}
	\begin{aligned}
		\label{eqn: eigenvalue recur 1}
		\dfrac{1}{\sigma_n}\lambda_{2k}^n
		=  & (4k+n-2)\left((4k+n-4)\dfrac{R^n_{2k}}{\sigma_n R^n_{2k-2}}\lambda^n_{2k-2}\right.\\
		&\left. - (4k+n-3)\dfrac{R^n_{2k}}{\sigma_{n+2}R^{n+2}_{2k-2}}\lambda^{n+2}_{2k-2}\right)
	\end{aligned}
\end{equation}
Note for $n\geq 1$ and $2k\geq 2$, we always have $4k+n-4\geq 1>0$. From the signs of the coefficients in the equation above, we see the following: If $\{\lambda_{2k}^n\}_{k=1}^{\infty}$ is a sequence of alternating signs indexed by even numbers $2k$, the signs of numbers in the sequence  $\{\lambda_{2k}^{n+2}\}$ will also be alternating. For the basic case $n=1$, the Riemannian manifold $\mathbb{RP}^1$ is isomorphic to $\mathbb{S}^1$ by doubling the angles between lines. For $\mathbb{S}^1$, the signs of eigenvalues of $\mathcal{K}$ corresponding to the eigenfunctions $\cos(m\theta)$ are alternating, depending on whether $m$ is odd or even \cite[section~2]{azevedo2015eigenvalues}. Thus, the signs of the sequence $\{\lambda^1_{2k} \}$ indexed by even $2k$ are also alternating. Combined with the argument above, we obtain the following result.
\begin{proposition}
	For odd number $n$, the MDS associated operator $T$ on $\mathbb{RP}^n$ has both infinitely many positive and negative eigenvalues. 
\end{proposition}
We then turn to the case of even-dimensional projective spaces, starting from the basic case of $\mathbb{RP}^2$. We expect that the signature of $T$ for even-dimensional projective spaces is similar to the odd dimension cases. A numeric computation for the first $6$ eigenvalues (corresponding to even degree spherical harmonics up to $2k=10$). From the pictures below, we can see that the first 6 eigenvalues have alternating signs.\\
\includegraphics[scale=0.5]{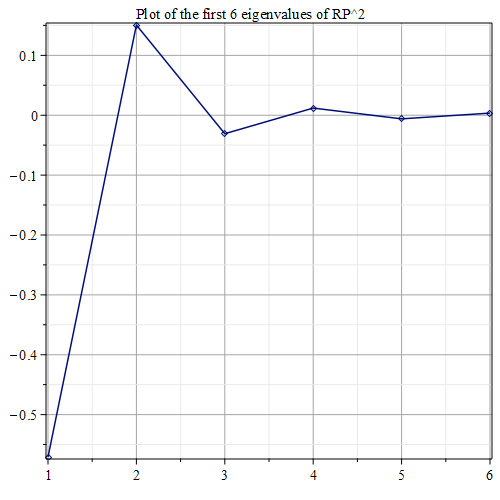}\\

Although computing the spectrum of $T$ for projective spaces can be difficult, we can easily obtain some information on $sgn(T)$ for $\mathbb{RP}^2$ using Mercer's theorem.
\begin{lemma}
	The operator $T$ on $\mathbb{RP}^2$ has infinitely many strictly positive and negative eigenvalues counting multiplicity.
\end{lemma}
\begin{proof}
	Since $\mathbb{RP}^2$ is homogeneous, we only need to prove this statement for $\mathcal{K}$.\\
	\\
	Suppose $\mathcal{K}$ has only finitely many strictly negative eigenvalues counting multiplicity. Let $\lambda_{2k}$ be the eigenvalue of $\mathcal{K}$ corresponding to the degree $2k$ spherical harmonics as before. Denote $\lbrace \phi_{2k}^m\rbrace$ with $-2k\leq m\leq 2k$ the standard $L_2$-orthonormal basis with respect to $\mu$ of Laplacian eigenspace $\mathcal{H}_{2k}$ of $\mathbb{RP}^2$ corresponding to degree $2k$ spherical harmonics.  The function
	\begin{align}
		N(x,y)=\sum_{\lambda_{2k}<0}\lambda_{2k} \sum_{m=-2k}^{2k}\phi^m_{2k}(x)\phi^m_{2k}(y)
	\end{align}
	is smooth and bounded, since the summation above is finite.
	We can decompose the integral kernel as $$K(x,y)=N(x,y)+H(x,y),$$ where $N(x,y)$ is a negative integral kernel and $H(x,y)$ is a positive integral kernel. Because the positive operator 
	$$H(f)(x)=\int_{\mathbb{RP}^2} f(y) d\mu(y)$$
	has a continuous bounded kernel, Mercer's theorem \cite[theorem~4.10]{Cucker2007} implies the following convergence with respect to summation of $\lambda_{2k} \phi_{2k}^m(x)\phi_{2k}^m(y)$ is absolute and uniform:
	\begin{align}
		H(x,y)= \sum_{\lambda_{2k}\geq 0} \sum_{m=-2k}^{2k}\lambda_{2k} \phi_{2k}^m(x)\phi_{2k}^m(y).
	\end{align}
	From the basics properties of spherical harmonics \cite[theorem~3.3.3]{Groemer1996}, we obtain
	\begin{align*}
		\sum_{m=-2k}^{2k} \phi_{2k}^m(x)\phi_{2k}^m(y)= &\mathrm{dim}(\mathcal{H}_k)P_{2k}(\cos (d(x,y)))\\
		= & (4k+1)P_{2k}(x\cdot y).
	\end{align*}
	Here we use $x,y$ to denote both points in $\mathbb{RP}^2$ and their lifts to $\mathbb{S}^2$. Since $P_{2k}$ is even, the value of $P_{2k}(x\cdot y)$ is always well-defined.
	\\
	\\
	Therefore, we obtain the following absolute and uniform convergence
	$$H(x,y)= \sum_{\lambda_{2k}\geq 0}\lambda_{2k}(4k+1)P_{2k}( x\cdot y)$$
	It follows that the following convergence of functions on $(\mathbb{RP}^2)^2$
	\begin{align}
		-\dfrac{1}{2}d^2(x,y)=-\dfrac{1}{2}\arccos^2(\vert x\cdot y\vert)=K(x,y)=\sum_{k=0}^{\infty} \lambda_{2k}(4k+1)P_{2k}(x\cdot y)
	\end{align}
	is absolute and uniform. 
	This implies for $t\in [0,1]$, the convergence
	\begin{align}
		\label{eqn:arcos con 1}
		\arccos^2(t)=-2\sum_{k=0}^{\infty}\lambda_{2k}(4k+1)P_{2k}(t)
	\end{align}
	is also absolute and uniform.\\
	\\
	On the other hand, we know from \eqref{eqn:proj spec}, the eigenvalues of $\mathcal{K}$ for $\mathbb{RP}^1$ are given by
	\begin{align}
		\lambda_{2j}^1=-\dfrac{1}{\pi}\int_0^1 \arccos^2(t) C_{2j}(t)\dfrac{1}{\sqrt{1-t^2}} dt,
	\end{align}
	where $C_{2j}$ are Chebyshev polynomials of the first kind of degree $2j$.
	Since the convergence in \eqref{eqn:arcos con 1} is absolute and uniform, dominated convergence theorem implies
	\begin{align}
		\label{eqn: eigenvalue s1}
		\lambda_{2j}^1=\dfrac{2}{\pi}\sum_{k=0}^{\infty}\lambda_{2k}(4k+1)\int_0^1 P_{2k}(t)C_{2j}(t)\dfrac{1}{\sqrt{1-t^2}} dt
	\end{align}
	According to \cite[p.~96]{Ismail2005}, the even degree Legendre polynomials can be expanded by Chebyshev polynomials as
	\begin{align*}
		P_{2k}(t)=\left(\dfrac{\Gamma(1/2+k)}{\Gamma(1/2)\Gamma(k+1) }\right)^2C_0+\dfrac{2}{\Gamma^2(1/2)}\sum_{i=1}^k\left( \dfrac{\Gamma(k-i+1/2)\Gamma(k+i+1/2)}{\Gamma(k-i+1)\Gamma(k+i+1)}\right)C_{2i}(t)
	\end{align*}
	The functions $C_{2j}(t)$ are even, and we have
	\begin{align*}
		\int_0^1 C_{2i}(t)C_{2j}(t)\dfrac{1}{\sqrt{1-t^2}}dt=\dfrac{\pi}{4}\left(\delta^i_j+\delta^0_i\delta^0_j\right)
	\end{align*}
	Therefore, we obtain
	\begin{align}
		\int_0^1 P_{2k}C_{2j}\dfrac{1}{\sqrt{1-t^2}} dt =0,\ 0\leq k<j\\
		\int_0^1 P_{2k}C_{2j}\dfrac{1}{\sqrt{1-t^2}} dt>0,\ k\geq j
	\end{align}
	These equalities and~\eqref{eqn: eigenvalue s1} imply $\lambda_{2j}^1$ is always non-negative for sufficiently large $j$. This leads to a contradiction. Hence, the operator $\mathcal{K}$ (hence $T$) for $\mathbb{RP}^2$ has to admit infinitely many negative eigenvalues counting multiplicity.
	
	Similarly, we can prove that $T$ on $\mathbb{RP}^2$ also has infinitely many positive eigenvalues counting multiplicity. This completes the proof.
\end{proof}
Starting from $\mathbb{RP}^2$, an induction argument on the dimension of the projective space using \eqref{eqn: eigenvalue recur 1} shows $T$ for every even dimension projective space has to admit both infinitely many positive and negative eigenvalues counting multiplicity. Combined with the result of odd-dimensional projective spaces, we obtain the following.
\begin{theorem}
	For any projective space $\mathbb{RP}^n$, the operator $T$ always has both infinitely many positive and negative eigenvalues counting multiplicity.
\end{theorem}
\subsubsection{Spectral asymptotics of $T$  odd-dimensional projective spaces}
Now we estimate the norm of the eigenvalues of $T$ for odd-dimensional projective spaces. This will show how the MDS maps of $\mathbb{RP}^n$ differ from the MDS maps of $\mathbb{S}^n$ for odd $n$. To estimate the asymptotics of the eigenvalues $\lbrace \lambda^n_{2k}\rbrace$, we need the following lemma.
\begin{lemma}
	For odd-dimensional projective space $\mathbb{RP}^n$, we have 
\[\lambda^n_{2k}=\Theta(k^{-(n+3)/2}). 
\]
\end{lemma}
\begin{proof}
	We prove this statement by induction on $n$. For the base case $n=1$, a simple integration by parts shows
	$$\lambda^1_{2k}=\Theta(k^{-2}).$$
	Suppose $\lambda^n_{2k}=\Theta(k^{-(n+3)/2})$ holds. All coefficients for the eigenvalues in~\eqref{eqn: eigenvalue recur 1} are non-zero. Solving this equation for $\lambda^n_{2k-2}$ and we get
	\begin{align*}
		\vert\lambda^{n+2}_{2k-2}\vert &\leq C_1(n,k)\left(\dfrac{R^{n+2}_{2k-2}}{R^n_{2k}} \right)\vert\lambda^n_{2k}\vert+C_2(n,k)\left( \dfrac{R^{n+2}_{2k-2}}{R^n_{2k-2}} \right)\vert \lambda^n_{2k-2}\vert \\
		\vert\lambda^{n+2}_{2k-2}\vert & C_2(n,k)\geq  C_2(n,k)\left( \dfrac{R^{n+2}_{2k-2}}{R^n_{2k-2}} \right)\vert \lambda^n_{2k-2}\vert, 
	\end{align*}
	where the second inequality comes from the fact that the sequence $\{\lambda^n_{2k}\}$ indexed by $2k$ has alternating signs. Here $C_1(n,k)$ and $C_2(n,k)$ are positive functions such that $C_1(n,k)=\Theta(k^{-2})$ and $C_2(n,k)=\Theta(1)$ when $n$ is fixed. Moreover, from the formula of Rodrigues constants, we get
	\begin{align}
		0<\dfrac{R^{n+2}_{2k-2}}{R^n_{2k}}=\Theta(k)\ \mathrm{and}, 
		0< \dfrac{R^{n+2}_{2k-2}}{R^n_{2k-2}}=\Theta(k^{-1}).
	\end{align}
	Then these inequalities together imply $\lambda^{n+2}_{2k}=\Theta(k^{-1-(n+3)/2})$.
\end{proof}

 Recall that by Lemma \ref{lemma:mds dist inv}, for any $x,y\in X$, the value of $\Vert\mathcal{M}(x)-\mathcal{M}(y)\Vert_2$ is independent of the choice of the eigenfunctions of $\mathcal{K}$ defining $\mathcal{M}$. 
 We are able to state the following result regarding the MDS maps for odd-dimensional projective spaces.

	\begin{theorem}
		For every odd $n>1$, the operator $T$ for $\mathbb{RP}^n$ is not trace-class. Furthermore, 
		given any $x\in\mathbb{RP}^n$, there exists a positive measure set $U_x$ such that
		\begin{equation}\label{eq-MDSdistinf1}
			\begin{aligned}
			\Vert \mathcal{M}(x)-\mathcal{M}(y)\Vert_2=+\infty,\quad \mbox{for all } y\in U_x,
		\end{aligned}
		\end{equation}		and the
		distance reconstruction formula~\eqref{eq_aeinj11} does not hold $\mu\otimes\mu$ almost everywhere. 
	\end{theorem}

\begin{proof}
According to \cite[theorem~3.1.4]{Groemer1996}, the dimension of the space of degree $2k$ spherical harmonics on $\mathbb{RP}^n$ is given by
\begin{align}
	dim(\mathcal{H}^n_{2k})=\dfrac{4k+n-1}{2k+n-1}\begin{pmatrix}
		2k+n-1\\
		n-1
	\end{pmatrix}.
\end{align}
Using the relation of binomial coefficients, we easily deduce that $dim(\mathcal{H}^n_{2k})$ is a polynomial in $k$ of degree $n-1$.
Then for all odd $n$, we have
\begin{align}
	\label{eqn:eigen k trace}
	\vert\lambda^n_{2k}\vert dim(\mathcal{H}^n_{2k})=\Theta(k^{(n-5)/2})
\end{align}
Therefore, for all odd $n$ with $n>1$, the operator $T$ on $\mathbb{RP}^n$ is not a trace-class operator.

	First we show $\Vert \mathcal{M}(x)-\mathcal{M}(y)\Vert_2=+\infty$ when $x\cdot y=0$. Let $\lbrace \phi^i_{2k}\rbrace$ with $i\in I^n_{2k}$ be an orthonormal basis with respect to $\mu$ of $\mathcal{H}^n_{2k}$. For $x,y\in\mathbb{RP}^n$, we have
\begin{equation}
	\begin{aligned}
		\label{eqn:mds dist perp}
		\Vert \mathcal{M}(x)-\mathcal{M}(y)\Vert^2_2 = & \sum_{\lambda_{2k}>0} \lambda_{2k}\sum_{i\in I^n_{2k}}(\phi^i_{2k}(x)-\phi^i_{2k}(y))^2\\
		= &\sum_{k=0}^{\infty}\lambda_{4k+2}\sum_{i\in I^n_{4k+2}}(\phi^i_{4k+2}(x)-\phi^i_{4k+2}(y))^2\\
		= &  \sum_{k=0}^{\infty}\lambda_{4k+2}\cdot  dim(\mathcal{H}^n_{4k+2})\cdot 2(1-P^n_{4k+2}(x\cdot y))
	\end{aligned}
\end{equation}
The Legendre polynomials have the following recurrence relation (see  \cite[proposition~3.3.11]{Groemer1996})
\begin{align}
	(k+n-1)P^n_{k+1}(t)-(2k+n-1)tP^n_k(t)+kP^n_{k-1}(t)=0.
\end{align}
Therefore, for fixed $n$ we can see $\{\vert P^n_{4k+2}(0)\vert\}$ for $k\geq 0$ is a decreasing sequence uniformly bounded away from $1$. Since $x\cdot y=0$, for any $n>1$, there exists some $\delta_n>0$ such that
\begin{align*}
	\Vert \mathcal{M}(x)-\mathcal{M}(y)\Vert^2_2\geq 2\delta_n \sum_{k=0}^{\infty}\lambda_{4k+2}\cdot dim(\mathcal{H}^n_{4k+2}).
\end{align*}
Using \eqref{eqn:eigen k trace}, we can see $\Vert \mathcal{M}(x)-\mathcal{M}(y)\Vert^2_2=+\infty$ for odd $n>1$.

From \eqref{eqn:mds dist perp}, we see that $\Vert \mathcal{M}(x)-\mathcal{M}(y)\Vert_2$ depends only on $\vert x\cdot y\vert$. Moreover, if $x_0\cdot x_1=0$, the triangle inequality implies there is a set of positive measure $U_{x_i}$ such that $\Vert \mathcal{M}(y)-\mathcal{M}(x_i)\Vert=+\infty$ for all $y\in U_{x_i}$ and for some $i=0$ or $1$. The operator $T$ commutes with the action of the isometry group $G$, so the distance $\Vert \mathcal{M}(x)- \mathcal{M}(y)\Vert_2$ is invariant under the action by $G$.
Since $G$ acts transitively on $\mathbb{RP}^n$, it follows that for every $x\in\mathbb{RP}^n$ with odd $n>1$, there exists a positive measure set $U_x$ such that~\eqref{eq-MDSdistinf1} holds.

Finally, by Fubini theorem one has then that the series on the left-hand side of the distance reconstruction formula~\eqref{eq_aeinj11} 
does not converge on a set of couples $(x,y)$ of positive $\mu\otimes\mu$ measure, and hence~\eqref{eq_aeinj11} 
does not hold $\mu\otimes\mu$ almost everywhere, which concludes the proof. 
\end{proof}

\bibliographystyle{plain}
\bibliography{MDSsymm-biblio0120final}

\begin{thebibliography}{10}

\bibitem{adams-blumstein-kassab2020MDS}
H.~Adams, M.~Blumstein, and L.~Kassab.
\newblock Multidimensional scaling on metric measure spaces.
\newblock {\em Rocky Mountain Journal of Mathematics}, 50(2):397--413, 2020.

\bibitem{azevedo2015eigenvalues}
D.~Azevedo and V.A. Menegatto.
\newblock Eigenvalues of dot-product kernels on the sphere.
\newblock {\em Proceeding Series of the Brazilian Society of Computational and
  Applied Mathematics}, 3(1), 2015.

\bibitem{Cucker2007}
Felipe Cucker and Ding~Xuan Zhou.
\newblock {\em Learning theory: an approximation theory viewpoint}, volume~24.
\newblock Cambridge University Press, 2007.

\bibitem{Groemer1996}
Helmut Groemer.
\newblock {\em Geometric applications of Fourier series and spherical
  harmonics}, volume~61.
\newblock Cambridge University Press, 1996.

\bibitem{heinonen2015sobolev}
J.~Heinonen, P.~Koskela, N.~Shanmugalingam, and J.T. Tyson.
\newblock {\em Sobolev spaces on metric measure spaces}.
\newblock Number~27. cambridge university press, 2015.

\bibitem{Helgason2001}
Sigurdur Helgason.
\newblock {\em Differential geometry and symmetric spaces}, volume 341.
\newblock American Mathematical Soc., 2001.

\bibitem{Helgason1984}
Sigurdur Helgason.
\newblock {\em Groups and geometric analysis: integral geometry, invariant
  differential operators, and spherical functions}, volume~83.
\newblock American Mathematical Society, 2022.

\bibitem{Ismail2005}
Mourad Ismail.
\newblock {\em Classical and quantum orthogonal polynomials in one variable},
  volume~13.
\newblock Cambridge university press, 2005.

\bibitem{kassab2019MDS}
Lara Kassab.
\newblock Multidimensional scaling: Infinite metric measure spaces.
\newblock {\em arXiv preprint arXiv:1904.07763}, 2019.
\newblock Masters thesis, Colorado State University.

\bibitem{Stepanov2022}
Alexey Kroshnin, Eugene Stepanov, and Dario Trevisan.
\newblock Infinite multidimensional scaling for metric measure spaces.
\newblock {\em ESAIM: Control, Optimisation and Calculus of Variations}, 28:58,
  2022.

\bibitem{levitin2023topics}
Michael Levitin, Dan Mangoubi, and Iosif Polterovich.
\newblock {\em Topics in spectral geometry}, volume 237.
\newblock American Mathematical Society, 2023.

\bibitem{lim2022classical}
Sunhyuk Lim and Facundo Memoli.
\newblock Classical multidimensional scaling on metric measure spaces.
\newblock {\em arXiv preprint arXiv:2201.09385}, 2022.

\bibitem{Muller}
Claus M{\"u}ller.
\newblock {\em Analysis of spherical symmetries in Euclidean spaces}, volume
  129.
\newblock Springer Science \& Business Media, 2012.

\bibitem{PuchSpokSteTrev-manif20}
Nikita P., Vladimir S., Eugene S., and Dario T.
\newblock Reconstruction of manifold embeddings into euclidean spaces via
  intrinsic distances, 2020.
\newblock arxiv 2012.13770.

\bibitem{Petersen1998}
Peter Petersen.
\newblock {\em Riemannian geometry (1st ed.)}, volume 171.
\newblock Springer-Verlag, 1998.

\bibitem{Sakai1978}
Takashi Sakai.
\newblock On the structure of cut loci in compact riemannian symmetric spaces.
\newblock {\em Mathematische Annalen}, 235:129--148, 1978.

\bibitem{Takeuchi1978}
Masaru Takeuchi.
\newblock On conjugate loci and cut loci of compact symmetric spaces i.
\newblock {\em Tsukuba journal of mathematics}, 2:35--68, 1978.

\bibitem{tyson2005characterizations}
J.T. Tyson and J.-M. Wu.
\newblock Characterizations of snowflake metric spaces.
\newblock {\em Ann. Acad. Sci. Fenn. Math}, 30(2):313--336, 2005.

\bibitem{vershik2023limit}
A~Vershik and F~Petrov.
\newblock Limit spectral measures of matrix distributions of metric triples.
\newblock {\em Functional Analysis and Its Applications}, 57:169--172, 2023.

\bibitem{von1941fourier}
J.~Von~Neumann and I.J. Schoenberg.
\newblock Fourier integrals and metric geometry.
\newblock {\em Transactions of the American Mathematical Society},
  50(2):226--251, 1941.

\bibitem{wang2012geometric}
J.~Wang.
\newblock {\em Geometric structure of high-dimensional data and dimensionality
  reduction}.
\newblock Springer, 2012.

\bibitem{wilson1935certain}
W.A. Wilson.
\newblock On certain types of continuous transformations of metric spaces.
\newblock {\em American Journal of Mathematics}, 57(1):62--68, 1935.

\bibitem{zelditch2017eigenfunctions}
Steve Zelditch.
\newblock {\em Eigenfunctions of the Laplacian on a Riemannian manifold},
  volume 125.
\newblock American Mathematical Soc., 2017.

\end{thebibliography}
\end{document}